\documentclass{amsart}

\usepackage{amsmath, amssymb, amsthm}

\newtheorem{thm}{Theorem}
\newtheorem{prop}[thm]{Proposition}
\newtheorem{lem}[thm]{Lemma}

\newcommand{\Cl}[1]{{\rm Cl}(#1)}

\numberwithin{equation}{section}

\begin{document}

\title[On centers of blocks]
       {On centers of blocks \\ with non-cyclic defect groups}
\author[Yoshihiro Otokita]
           {Yoshihiro Otokita}
\address{Yoshihiro Otokita: \newline
               Department of Mathematics and Informatics \newline
               Graduate School of Science \newline
               Chiba University}
\email{otokita@chiba-u.jp}

\begin{abstract}
In this short note we study the center $ZB$ of a block $B$ of a finite group over an algebraically closed field of prime characteristic through its Loewy length ${\rm ll}ZB$. A result of Okuyama in 1981 gave an upper bound for ${\rm ll}ZB$ in terms of defect group of $B$. The purpose of this note is to improve this bound for non-cyclic defect groups.
\end{abstract}

\maketitle

\section{Introduction}
In this short note we study the center of a block of a finite group over an algebraically closed field of prime characteristic through its Loewy length.

Let $G$ be a finite group and $F$ an algebraically closed field of characteristic $p>0$. For a block $B$ of the group algebra $FG$ we denote by $ZB$ its center. In order to examine the structure of $ZB$ we use its Loewy length ${\rm ll}ZB$, that is, the nilpotency index of the Jacobson radical $JZB$. A result of Okuyama \cite{Ok} states that ${\rm ll}ZB \le p^{d}$ where $d$ is the defect of $B$. In addition Koshitani-K\"{u}lshammer-Sambale \cite{KKS} determines ${\rm ll}ZB$ for cyclic defect groups. By this, we consider the other cases in this note. More precisely, we prove the following theorem:

\begin{thm} \label{Thm1} Let $B$ be a block of $FG$ with non-cyclic defect group of order $p^{d}$. Then 
\[ {\rm ll}ZB \le p^{d-1} + p - 1. \]
\end{thm}

Now let us take a block $B$ with defect group $D$ of order $3^{5}$ as an example. Then ${\rm ll}ZB \le 243$ by Okuyama's formula. Theorem \ref{Thm1} implies that $D$ is cyclic provided $84 \le {\rm ll}ZB \le 243$. In this case ${\rm ll}ZB = 243$ or $122$ by \cite{KKS}. In all other cases we have ${\rm ll}ZB \le 83$. 

We remark that the converse of this theorem is not true in general. For instance a block $B$ with cyclic defect group $C_{p^{2}}$ and inertial quotient group $C_{p-1}$ satisfies ${\rm ll}ZB = p+2 \le 2p-1$ whenever $p \ge 3$. 

\section{Preliminaries}
We prepare some notations. For a conjugacy class $C \in \Cl{G}$ its defect group $\delta (C)$ is defined as a Sylow $p$-subgroup of $C_{G}(x)$ where $x \in C$. For a $p$-subgroup $P$ of $G$ we set
\[ I_{G}(P) = \sum_{C \in \Cl{G}, \delta (C) \le_{G} P} FC^{+}, \ \ \ 
  \widetilde{I}_{G}(P) = \sum_{C \in \Cl{G}, \delta (C) <_{G} P} FC^{+} \]
where $C^{+}$ is the class sum of $C$. These are ideals of the center $ZFG$ of $FG$. Furthermore we denote by $t(P)$ the Loewy length of $FP$ following Wallace \cite{Wa}. 
Here we refine a lemma in Passman \cite{Pa}.

\begin{lem} \label{Lem2} Let $P$ be a $p$-subgroup of $G$. Then the following hold:
\begin{enumerate}
\item $I_{G}(P) \cdot JZFG^{t(Z(P))} \subseteq \widetilde{I}_{G}(P)$.
\item $I_{G}(P) \cdot JZFG^{(p^{a+1} - 1 / p - 1)} = 0$ where $|P| = p^{a}$.
\end{enumerate}
\end{lem}
\begin{proof}
It remains only to prove (1) by {\cite[Lemma 3 (ii)]{Pa}}. Let ${\rm Br}_{P} : ZFG \to ZFC_{G}(P)$ be the Brauer homomorphism associated to $P$. Since ${\rm Br}_{P}$ maps nilpotent elements to nilpotent elements we have ${\rm Br}_{P} (JZFG) \subseteq JZFC_{G}(P)$. On the other hand ${\rm Br}_{P} (I_{G}(P)) \subseteq I_{C_{G}(P)}(Z(P))$ holds (see the proof of {\cite[Lemma 3 (i)]{Pa}}). Thus it follows from {\cite[Lemma 2 (i)]{Pa}} that 
\begin{align*}
{\rm Br}_{P} (I_{G}(P) \cdot JZFG^{t(Z(P))}) & \subseteq I_{C_{G}(P)} (Z(P)) \cdot JZFC_{G}(P)^{t(Z(P))} \\
                                                                  & = JFZ(P)^{t(Z(P))} \cdot I_{C_{G}(P)}(Z(P)) = 0 
\end{align*}
since $Z(P)$ is central in $C_{G}(P)$. Therefore we deduce
\[ I_{G}(P) \cdot JZFG^{t(Z(P))} \subseteq {\rm Ker} {\rm Br}_{P} \cap I_{G}(P) = \widetilde{I}_{G}(P) \]
as claimed.
\end{proof}

\section{Proof of main theorem}

We first improve K\"{u}lshammer-Sambale {\cite[Theorem 12 and Proposition 15]{KS}} by using Lemma \ref{Lem2}. 

\begin{prop} \label{Prop3} Let $B$ be a block of $FG$ with non-abelian defect group of order $p^{d}$. Then ${\rm ll}ZB < p^{d-1}$.
\end{prop}

\begin{proof}
We may assume $p \neq 2$ by {\cite[Proposition 15]{KS}}. Let $D$ be a defect group of $B$. If $Z(D)$ is cyclic of order $p^{d-2}$ then $D$ is one of the following types:
\begin{align*}
& M_{p^{d}} := <x, y \ | \ x^{p^{d-1}} = y^{p} = 1, y^{-1}xy = x^{p^{d-2}+1} >, \\
& W_{p^{d}} := <x, y, z \ | \ x^{p^{d-2}} = y^{p} = z^{p} = [x, y] = [x, z] = 1, [y, z] = x^{p^{d-3}}>
\end{align*}
where $d \ge 3$. In both cases ${\rm ll}ZB < p^{d-1}$ (see {\cite[Proposition 10 and Lemma 11]{KS}}). If $D$ has a cyclic subgroup of index $p$ then $D \simeq M_{p^{d}}$ (e.g. see {\cite[Chapter 5, Theorem 4.4]{Go}} ). Thereby we need only consider the other cases. Since $Z(D)$ is non-cyclic or has order at most $p^{d-3}$, we have $\lambda_{0} := p^{d-3} + p - 1 \ge t(Z(D))$. Thus we first obtain from Lemma \ref{Lem2} (1) that
\[ I_{G}(D) \cdot JZFG^{\lambda_{0}} \subseteq I_{G}(D) \cdot JZFG^{t(Z(D))} \subseteq \widetilde{I}_{G}(D) = \sum_{D_{1} < D} I_{G}(D_{1}). \]
By our assumptions above, $D_{1}$ is non-cyclic or has order at most $p^{d-2}$. In both cases we have $\lambda_{1} := p^{d-2}+p-1 \ge t(Z(D_{1}))$. Thus
\begin{align*}
I_{G}(D) \cdot JZFG^{\lambda_{0} + \lambda_{1}} & \subseteq \sum_{D_{1} < D} I_{G}(D_{1}) \cdot JZFG^{\lambda_{1}} \subseteq \sum_{D_{1} < D} I_{G}(D_{1}) \cdot JZFG^{t(Z(D_{1}))} \\
& \subseteq \sum_{D_{1} < D} \widetilde{I}_{G}(D_{1}) = \sum_{D_{2} < D_{1} < D} I_{G}(D_{2}). 
\end{align*}
Finally,  it follows from Lemma \ref{Lem2} (2) that 
\[ I_{G}(D) \cdot JZFG^{\lambda_{0} + \lambda_{1} + \lambda_{2}} \subseteq \sum_{D_{2} < D_{1} < D} I_{G}(D_{2}) \cdot JZFG^{\lambda_{2}} = 0 \]
where $\lambda_{2} := p^{d-1}-1/p-1$ since $|D_{2}| \le p^{d-2}$. Now let $e$ be the block idempotent of $B$. Then
\[ JZB^{\lambda_{0} + \lambda_{1} + \lambda_{2}} = e JZFG^{\lambda_{0} + \lambda_{1} + \lambda_{2}} \subseteq I_{G}(D) \cdot JZFG^{\lambda_{0} + \lambda_{1} + \lambda_{2}} = 0 \]
and this means ${\rm ll}ZB \le \lambda_{0} + \lambda_{1} + \lambda_{2}$. Accordingly, ${\rm ll}ZB < p^{d-1}$ except for one case that $p=3$ and $d=4$. Hence we consider this case in the following. From \cite{Ot2} (see {\cite[proof of Theorem 1.3]{Ot1}}), there exists a non-trivial $B$-subsection $(u, b)$ such that 
\[ {\rm ll}ZB \le (|u|-1) {\rm ll}Z\bar{b} + 1 \]
where $\bar{b}$ is the unique block of $F[C_{G}(u) / \langle u \rangle]$ dominated by $b$. We may assume that $\bar{b}$ has defect group $C_{D}(u) / \langle u \rangle$ by Sambale {\cite[Lemma 1.34]{Sa}}. We put $|u|=3^{s}$ and $|C_{D}(u)|=3^{r}$. If $r \le d-2$ then 
\[ {\rm ll}ZB \le (3^{s}-1)3^{r-s} + 1 \le (3^{s}-1)3^{d-s-2} + 1 < 3^{d-1}. \]
In case of $r=d-1$, we may $r > s$ by our assumptions and thus 
\[ {\rm ll}ZB \le (3^{s}-1)3^{r-s} + 1 = 3^{r} - 3^{r-s} + 1 < 3^{r} = 3^{d-1} \]
as required. Therefore we may assume $d=r$, so that $u \in Z(D)$. Hence $|u|=3$ and $D / \langle u \rangle$ is isomorphic to $C_{3} \times C_{3} \times C_{3}, C_{9} \times C_{3}, M_{27}$ or $W_{27}$ by our assumptions. In all cases ${\rm ll}Z \bar{b} \le11$ by {\cite[Theorem 1, Proposition 10 and Lemma 11]{KS}}. Consequently, ${\rm ll}ZB \le 23 < 27 = p^{d-1}$. Our claim is completely proved.
\end{proof} 

Theorem \ref{Thm1} is an immediate corollary to Proposition \ref{Prop3}.

\begin{proof}[Proof of Theorem \ref{Thm1}]
We may assume $p^{d-1} < {\rm ll}ZB$ and thus $D$ is abelian by Proposition \ref{Prop3}. In this case K\"{u}lshammer-Sambale \cite{KS} has proved that ${\rm ll}ZB \le t(D)$. Hence our claim follows.
\end{proof}


\begin{thebibliography}{99}

\bibitem{Go}
D. Gorenstein,
\textit{Finite groups},
Harper \& Row Publishers, New York (1968).

\bibitem{KKS}
S. Koshitani, B. K\"{u}lshammer, B. Sambale,
\textit{On Loewy lengths of blocks},
Math. Proc. Cambridge Philos. Soc. \textbf{156} (2014), 555-570.

\bibitem{KS}
B. K\"{u}lshammer, B. Sambale,
\textit{Loewy lengths of centers of blocks},
arXiv:1607.06241v2.

\bibitem{Ok}
T. Okuyama,
\textit{On the radical of the center of a group algebra},
Hokkaido Math. J. \textbf{10} (1981), 406--408.

\bibitem{Ot1}
Y. Otokita,
\textit{Some studies on Loewy lengths of centers of $p$-blocks},
arXiv:1605.07949v2.

\bibitem{Ot2}
Y. Otokita,
\textit{Characterizations of blocks by Loewy lengths of their centers},
Proc. Amer. Math. Soc., in press.

\bibitem{Pa}
D. S. Passman,
\textit{The radical of the center of a group algebra},
Proc. Amer. Math. Soc. \textbf{78} (1980), 323--326.

\bibitem{Sa}
B. Sambale,
\textit{Blocks of finite groups and their invariants},
Springer Lecture Notes in Math., Vol. 2127, Springer-Verlag, Cham, 2014.

\bibitem{Wa}
D. A. R. Wallace,
\textit{Lower bounds for the radical of the group algebra of a finite $p$-soluble group},
Proc. Edinburgh Math. Soc. (2) \textbf{16} (1968/69), 127--134.

\end{thebibliography}
\end{document}